\documentclass[a4paper]{amsart}
\usepackage{amsmath,amssymb,amsthm,xcolor,graphicx}
\usepackage[allcolors=purple,citecolor=violet,colorlinks=true]{hyperref}
\usepackage[margin=3.2cm]{geometry}

\newtheorem{theorem}{{Theorem}}[section]
\newtheorem{proposition}[theorem]{{Proposition}}
\newtheorem{definition}[theorem]{{Definition}}
\newtheorem{lemma}[theorem]{{Lemma}}
\newtheorem{corollary}[theorem]{{Corollary}}
\newtheorem{remark}[theorem]{{Remark}}
\newtheorem{conjecture}[theorem]{{Conjecture}}
\newtheorem{example}[theorem]{{Example}}

\newtheorem*{Lconjecture}{{Conjecture \ref{conj:Lich2}}}

\newcommand{\N}{\mathbb{N}}
\newcommand{\Z}{\mathbb{Z}}

\newcommand{\R}{\mathbb{R}}
\newcommand{\C}{\mathbb{C}}
\newcommand{\bS}{\mathbb{S}}
\newcommand{\Id}{\operatorname{Id}}
\newcommand{\Eins}{\mathsf{Eins}}
\newcommand{\hEins}{\widehat{\mathsf{Eins}}{}}
\newcommand{\SO}{\mathsf{SO}}
\newcommand{\OG}{\mathsf{O}}
\newcommand{\hOG}{\widehat{\mathsf{O}}}
\newcommand{\GL}{\mathsf{GL}}
\newcommand{\Mat}{\mathsf{Mat}}
\newcommand{\Isom}{\mathrm{Isom}}
\newcommand{\Conf}{\mathrm{Conf}}

\newcommand{\Diff}{\mathrm{Diff}}
\newcommand{\Aut}{\mathrm{Aut}}
\newcommand{\CB}{\mathcal{C}}
\newcommand{\pr}{\pi_*}

\begin{document}

%
\title[Conformal transformations of spacetimes]{Conformal transformations of spacetimes without observer horizons}
\date{\today}

\author{Leonardo Garc\'ia-Heveling}
\address{Fachbereich Mathematik, Universit\"at Hamburg, Germany}
\email{\href{mailto:leonardo.garcia@uni-hamburg.de}{leonardo.garcia@uni-hamburg.de}}
\urladdr{\url{https://leogarciaheveling.github.io/}}
\author{Abdelghani Zeghib}
\address{UMPA, CNRS, \'Ecole Normale Sup\'erieure de Lyon, France}
\email{\href{mailto:abdelghani.zeghib@ens-lyon.fr}{abdelghani.zeghib@ens-lyon.fr}}
\urladdr{\url{http://www.umpa.ens-lyon.fr/~zeghib/}}

\thanks{We thank Charles Frances for discussions about conformal transformations of Minkowski spacetime and the Einstein universe.}

%
%

%
%

%
%

\begin{abstract}
 We prove that for a certain class of Lorentzian manifolds, namely causal spacetimes without observer horizons, conformal transformations can be classified into two types: escaping and non-escaping. This means that successive powers of a given conformal transformation will either send all points to infinity, or none. As an application, we classify the conformal transformations of Einstein's static universe. We also study the question of essentiality in this context, i.e.\ which conformal transformations are isometric for some metric in the conformal class.
\end{abstract}

\maketitle

\setcounter{tocdepth}{1}

\tableofcontents

\section{Introduction}

In Lorentzian geometry and general relativity, the causal order on spacetime $(M,g)$ plays a fundamental role. Two points $p,q \in M$ are called causally related if there exists a future-directed causal curve joining them. That is, a curve with tangent vector in the future lightcone, and thus a possible trajectory for a particle. By a celebrated result of Hawking and Malament (see Theorem~\ref{thm:Haw-Mal} below), the bijections $\phi \colon M \to M$ preserving the causal order are the conformal transformations of $M$, i.e.\ diffeomorphisms such that $\phi^* g = \Omega g$ for some positive function $\Omega$. In the present paper, we exploit this fact to gain knowledge of the conformal group $\Conf^\uparrow(M,g)$, an object of importance more broadly in differential geometry. The superscript $\uparrow$ indicates that we restrict to \emph{time-orientation preserving} transformations (meaning that the future cone is mapped to the future cone).

In \cite{GaZe}, we followed a similar approach to study the group of time-orientation preserving isometries, and found that for causal spacetimes without observer horizons, it splits as a semi-direct product
\begin{equation} \label{eq:isomsplit}
 \Isom^\uparrow(M,g) = L \ltimes N,
\end{equation}
where $N$ is a subgroup that leaves a time function invariant, while $L$ acts by time-translations. The condition of ``no observer horizons'' means that the past and future sets $I^\pm(\gamma)$ of every inextendible causal curve $\gamma$ equal the whole spacetime $M$. It implies that the spacetime has a compact Cauchy surface. The proof of \eqref{eq:isomsplit} relies on two fundmental tools: preservation of the causal relation by isometries, and properness of the action on the orthonormal frame bundle. For conformal transformations, the first tool is still available, but the second one is not. We obtain the following result, which tells us that conformal transformations can still be classified in two types, similarly to \eqref{eq:isomsplit}, but this classification does not respect the group structure (i.e.\ does not lead to a semi-direct product splitting). We state it here in simplified form; see Theorem~\ref{thm:escap} for additional details.

\begin{theorem} \label{thm:intro1}
 Let $(M,g)$ be a causal spacetime satisfying the no observer horizons condition, and let $\phi \colon M \to M$ be a time-orientation preserving conformal transformation. Then, there are three mutually exclusive possibilities:
 \begin{enumerate}
  \item[(i$^+$)] For every point $p \in M$ and every Cauchy time function $\tau$,
  \begin{equation*}
   \tau(\phi^k(p)) \to +\infty \text{ as }k \to \infty.
  \end{equation*}
  \item[(i$^-$)] For every point $p \in M$ and every Cauchy time function $\tau$,
  \begin{equation*}
   \tau(\phi^k(p)) \to -\infty \text{ as }k \to \infty.
  \end{equation*}
  \item[(ii)] For every point $p \in M$, the sequence $\left(\phi^k(p)\right)_{k \in \N}$ stays inside a compact set.
 \end{enumerate}
 In cases (i$^+$) and (i$^-$), the action of $\phi$ is free, proper, and cocompact.
\end{theorem}

We say that $\phi$ is \emph{future-escaping} if (i$^+$) holds, \emph{past-escaping} if (i$^-$) holds, and \emph{non-escaping} if (ii) holds. These notions also pass to the Lie algebra (see Proposition~\ref{prop:conj}). We simply say that $\phi$ is \emph{escaping} if it is future- or past-escaping. The key fact is that the behavior does not depend on the point. This clearly fails on spacetimes with observer horizons, such as the Minkowski and de Sitter spacetimes. There, boosts are examples of conformal transformations which have a fixed point, but send other points to $\pm \infty$ when applied repeatedly.

We then turn to study the Einstein static universe $\hEins^{1,n}$. This is the universal (if $n\geq2$) cover of the compact Einstein universe, which in turn is the conformal compactification of Minkowski spacetime $\R^{1,n}$. The Einstein static universe is isometric to the Lorentzian product $\R \times \bS^n$ (where $\R$ is the time direction and $\bS^n$ the round sphere). Having such a symmetric form, $\hEins^{1,n}$ has a large conformal group $\Conf^\uparrow(\hEins^{1,n}) = \hOG^\uparrow(2,n+1)$, which has been extensively studied (see, for instance, \cite{PeRi,SchCFT,EMN}). To illustrate the power of Theorem~\ref{thm:intro1}, we perform a ``manual'' classification of elements of $\hOG^\uparrow(2,n+1)$, which turns out to be more involved that the ``abstract'' proof of Theorem~\ref{thm:intro1}. This also provides us some more refined information about $\hOG^\uparrow(2,n+1)$.

\begin{theorem} \label{thm:intro2}
 In the group $\hOG^\uparrow(2,n+1)$ of time-orientation preserving conformal transformations of $\hEins^{1,n}$, the following statements hold.
 \begin{enumerate}
  \item All future-escaping elements in $\hOG^0(2,n+1)$ are smoothly conjugate to each other. The same holds for all past-escaping elements.
  \item Any non-escaping element has infinitely many fixed points, or preserves a time function (or both).
 \end{enumerate}
\end{theorem}

Here $\hOG^0(2,n+1)$ denotes the connected component of the identity, and $\phi,\psi$ are smoothly conjugate if there is a diffeomorphism $\chi$ (not necessarily conformal) such that $\phi = \chi \circ \psi \circ \chi^{-1}$. To prove (i), we show that the quotients $M / \phi$ and $M / \psi$ are diffeomorphic (but not necessarily conformally diffeomorphic). Smooth conjugacy is thus an empty condition on escaping elements. We expect that the situation is the opposite for non-escaping elements: there, differential conjugacy is equivalent to conformal conjugacy. This is the subject of a work in progress \cite{AGZ}.

The last part of the paper deals with the question of essentiality. We say that $\Conf^\uparrow(M,g)$ is essential if there does not exist $\Omega \colon M \to (0,\infty)$ such that $\Conf^\uparrow(M,g) = \Isom^\uparrow(M,\Omega g)$. Similarly, we say that a transformation $\phi \in \Conf^\uparrow(M,g)$ is essential if there does not exist $\Omega$ such that $\phi \in \Isom^\uparrow(M,\Omega g)$. Note that admitting one essential transformation implies having essential conformal group, but not the other way around. Essentiality is well-understood in the Riemannian case, thanks to a conjecture of Lichnerowicz, proven by Ferrand \cite{Fer69,Fer} and Obata \cite{Oba}. The conclusion is that the only Riemannian manifolds with essential conformal group are those conformal to $\bS^n$ or $\R^n$. Here, we propose an analogous conjecture for spacetimes without observer horizons.

\begin{conjecture} \label{conj:Lich2}
 Let $(M,g)$ be a causal spacetime satisfying the NOH. Then $\Conf^\uparrow(M,g)$ is essential if and only if $(M,g)$ is conformally diffeomorphic to the Einstein static universe $\hEins^{1,n}$.
\end{conjecture}

See Section~\ref{sec:essentiality} and references therein for previous generalizations of the Lichnerowicz conjecture, including the (still largely open) case of compact Lorentzian manifolds. There we also provide evidence for our conjecture by proving the following.
\begin{itemize}
 \item $\Conf^\uparrow(M,g)$ is essential if and only if it acts non-properly on $M$ (Proposition~\ref{prop:inessential}).
 \item If $\phi \in \Conf^\uparrow(M,g)$ is essential, then it preserves an achronal set (Proposition~\ref{prop:achronalset}).
 \item If $(M,g)$ is conformally flat with finite fundamental group and admits an essential transformation, then $(M,g)$ is conformal to $\hEins^{1,n}$ (Theorem~\ref{thm:finitepi1}).
\end{itemize}
Similar versions of these statements were used in the proof of the Riemannian Lichnerowicz conjecture. We also comment on the possibility of a Lichnerowicz conjecture for orbifolds, which necessarily needs to be more nuanced than for manifolds (see Remark~\ref{rem:orbi}).

\section{Preliminaries}

\subsection{Lorentzian geometry}

Let $(M,g)$ be a spacetime, that is, a time-oriented Lorentzian manifold (sometimes, we will abbreviate it to $M$). We employ the usual notation, where $\ll$ and $\leq$ denote the chronological and causal relations, and $I^\pm( \cdot)$ and $J^\pm(\cdot)$ denote the chronological and causal futures and pasts of points, sets, and curves.

\begin{definition}
 A spacetime $M$ is called
 \begin{enumerate}
  \item \emph{Causal} if $\leq$ is antisymmetric (equivalently, if there are no closed causal curves.).
  \item \emph{Distinguishing} if
  \begin{equation*}
   I^-(p) = I^-(q) \iff p=q \iff I^+(p) = I^+(q).
  \end{equation*}
  \item \emph{Globally hyperbolic} if $M$ is causal and for all $p,q \in M$, $J(p,q) := J^+(p) \cap J^-(q)$ is compact.
 \end{enumerate}
\end{definition}

Recall also that a Cauchy surface is a set $\Sigma \subset M$ that is intersected exactly once by every inextendible causal curve, and $\tau \colon M \to \R$ is called a Cauchy temporal function if its gradient vector $d\tau^\sharp$ is everywhere past-directed timelike (then, each $\tau$-level set is a Cauchy surface). The following classical theorem relates these concepts.

\begin{theorem}[{\cite{Ger,BeSa1}}] \label{thm:GerBeSa}
 Let $(M,g)$ be a spacetime. The following are equivalent:
 \begin{enumerate}
  \item $(M,g)$ is globally hyperbolic.
  \item $(M,g)$ contains a Cauchy surface.
  \item $(M,g)$ admits a Cauchy temporal function.
 \end{enumerate}
 Moreover, any two Cauchy hypersurfaces are homeomorphic, and a choice of Cauchy temporal function $\tau$ defines a diffeomorphism $F \colon \R \times \Sigma \to M$ for $\Sigma = \tau^{-1}(0)$, such that
 \begin{equation*}
  F^* g_{(t,x)} = - \beta(t,x) dt^2 + (h_t)_x,
 \end{equation*}
 where $\beta > 0$ and $h_t$ is a Riemannian metric on $\{t\} \times \Sigma$.
\end{theorem}

The following causality condition plays a central role in our paper.

\begin{definition} \label{def:NOH}
 A spacetime $(M,g)$ is said to satisfy the \emph{No Observer Horizons condition (NOH)} if $I^+(\gamma) = I^-(\gamma) = M$ for every (future- and past-) inextendible causal curve $\gamma$.
\end{definition}

In Definition \ref{def:NOH}, it is equivalent to require the same conditions only on timelike curves. The NOH is also equivalent to the condition that the future and past causal boundaries of $M$ each consist of a single element.

\begin{definition}[{\cite[p.~3]{Paeng}}]
 A pair of subsets $A,B \subset M$ is called \emph{totally timelike connected} if for all points $p \in A$ and $q \in B$ it holds that $p \ll q$.
\end{definition}

Note that the order matters, since in particular, it follows that $B$ is in the future of $A$. Moreover, total timelike connectedness defines a transitive relation on subsets. The following lemma summarizes causality theory results proven in \cite[Sec. 2.3]{GaZe}.

\begin{lemma} \label{lem:NOHCauchy}
 Suppose $(M,g)$ is a causal spacetime that satisfies the NOH. Then $(M,g)$ is globally hyperbolic with compact Cauchy surfaces. Moreover, for every Cauchy temporal function $\tau \colon M \to \R$ and every $s \in \R$, there exist $t_1 < s < t_2$ such that $\tau^{-1}(t_1)$, $\tau^{-1}(s)$ and $\tau^{-1}(t_2)$ are pairwise totally timelike connected, in that order.
\end{lemma}

A conformal transformation of $(M,g)$ is a diffeomorphism $\phi \colon M \to M$ such that $\phi^* g = \Omega g$, for $\Omega$ a smooth positive function on $M$. If, furthermore, the differential of $\phi$ maps future-directed causal vectors to future-directed causal vectors, we say that $\phi$ preserves the time orientation of $(M,g)$. We denote by $\Conf^\uparrow(M,g)$ the Lie group of time-orientation preserving conformal transformations.

The following result due to Hawking \cite[Lem.~19]{HawAdams}, and improved by Malament \cite[Thm.~2]{Mal}, physically motivates the study of conformal transformations, as they are precisely the maps preserving the causal structure.

\begin{theorem} \label{thm:Haw-Mal}
 Let $M,N$ be two distinguishing spacetimes, and $\phi \colon M \to N$ a bijection. Then, the following are equivalent:
 \begin{enumerate}
  \item $\phi$ and $\phi^{-1}$ preserve the chronological relations, i.e.
  \begin{equation*}
   x \ll y \iff \phi(x) \ll \phi(y).
  \end{equation*}
  \item $\phi$ is a time-orientation preserving conformal diffeomorphism.
 \end{enumerate}

\end{theorem}

\subsection{Group actions}

An action of a group $G$ on a manifold $M$ is a map $G \times M \to M,\ (g,p) \mapsto g(p)$ such that for all $p \in M$, $g(h(p)) = (gh)(p)$ for all $g,h \in G$ and $e(p) = p$ for $e \in G$ the identity element. Of particular importance are proper actions, since this condition implies that the dynamics are non-chaotic, and one can obtain a Hausdorff quotient. The following proposition gives different characterizations of properness.

\begin{proposition}[{\cite[Prop.~21.5]{Lee}}] \label{prop:properactions}
 Let $M$ be a smooth manifold and $G$ a Lie group acting continuously on $M$. The following are equivalent:
 \begin{enumerate}
  \item The action is proper, meaning that
  \begin{equation*}
   G \times M \longrightarrow M \times M,\quad (g,p) \longmapsto (g(p),p)
  \end{equation*}
  is a proper map.
  \item For every pair of sequences $(p_i)_i$ in $M$ and $(g_i)_i$ in $G$ such that $(p_i)_i$ and $(g_i(p_i))_i$ converge in $M$, $(g_i)_i$ has a convergent subsequence in $G$.
  \item For every compact set $K \subseteq M$, the set
  \begin{equation*}
   G_K := \{ g \in G \mid K \cap g(K) \neq \emptyset \}
  \end{equation*}
 is compact. 
 \end{enumerate}
\end{proposition}

We will say that an element $g \in G$ acts properly if the action of the subgroup $\langle g \rangle$ generated by $g$ is proper. The following lemma will be useful in various proofs, because it allows us to consider powers of a given element, instead of the element itself.

\begin{lemma} \label{lem:powers}
 Let $G$ be a Lie group acting smoothly on a manifold $M$, and let $N \in \N$ and $g \in G$. If $\langle g^N \rangle$ acts properly, then $\langle g \rangle$ acts properly. The same holds for the closures $\overline{\langle g^N \rangle}$ and $\overline{\langle g \rangle}$.
\end{lemma}

\begin{proof}
 Let $(n_k)_{k \in \N}$ and $(p_k)_{k \in \N}$ be a sequence of integers and of points in $M$, respectively. Suppose that $p_k \to p$ and $g^{n_k}(p_k) \to q$ as $k \to \infty$. We need to show that a subsequence of $(g^{n_k})_{k \in \N}$ converges in $\langle g \rangle$. Necessarily, there is a $R \in \{0,...,\vert N-1 \vert \}$ such that $(n_k)_{k \in \N}$ contains a subsequence of the form
 \begin{equation*}
  (N i_k + R)_{k \in \N},
 \end{equation*}
 where $i_k \in \Z$ (to see this, notice that $(n_k \mod N)_{k \in \N}$ must have at least one element repeated infinitely many times). Now
 \begin{equation*}
  g^R(p_k) \to g^R(p) \quad \text{and} \quad g^{N i_k}(g^R(p_k)) = g^{Ni_k + R} (p_k) \to q,
 \end{equation*}
 so it follows by properness of $g^N$ that $(g^{Ni_k})_{k \in \N}$ has a converging subsequence $(h_l)_{l \in \N}$ in $\langle g^N \rangle$. But then $(g^R \circ h_l)_{l \in \N}$ is a converging subsequence of $(g^{n_k})_{k \in \N}$ in $\langle g \rangle$, and we are done.

 To prove the statement for the closures $\overline{\langle g^N \rangle}$ and $\overline{\langle g \rangle}$, we use the following fact: Every element of $\overline{\langle g \rangle}$ is a commuting product $g^R h$ for $0 \leq R < N$ and $h \in \overline{\langle g^N \rangle}$ (this can be shown by an argument with sequences, similar to the previous paragraph). It follows that for every compact $K \subset M$, the set
 \begin{equation*}
  K' := \bigcup_{R = 0}^{N-1} g^R(K)
 \end{equation*}
 is also compact (in $M$), and that
 \begin{equation*}
  \{h \in \overline{\langle g \rangle} \mid K \cap h(K) \neq \emptyset  \}
 \end{equation*}
  is compact in $\overline{\langle g \rangle}$ if and only if
 \begin{equation*}
  \{f \in \overline{\langle g^N \rangle} \mid K' \cap f(K') \neq \emptyset \}
 \end{equation*}
 is compact in $\overline{\langle g^N \rangle}$. Thus $\overline{\langle g^N \rangle}$ acting properly implies that $\overline{\langle g \rangle}$ acts properly, as desired.
\end{proof}

Finally, recall the following terminology: a group action is called \emph{free} if it has no fixed points. When a Lie group $G$ acts smoothly, freely and properly on a manifold $M$, then the quotient space $M / G$ is a manifold. We say that the action is \emph{cocompact} if $M / G$ is compact.


\section{Escaping vs non-escaping conformal transformations}

In this section, we restate and prove our main result. Theorem~\ref{thm:intro1} from the introduction easily follows, given that, by Lemma~\ref{lem:NOHCauchy}, sets of the form $\tau^{-1}([a,b])$ are compact, for every $-\infty < a \leq b < +\infty$ and every Cauchy time function $\tau$.

\begin{theorem} \label{thm:escap}
 Let $(M,g)$ be a causal spacetime satisfying the NOH, and let $\phi \in \Conf^\uparrow(M,g)$ be a time-orientation preserving conformal transformation. Then, one of the following mutually exclusive possibilities holds:
 \begin{enumerate}
  \item For every compact subset $K \subset M$, $\displaystyle \bigcup_{j \in \Z} \phi^j(K)$ is relatively compact. 
  \item For a fixed $J \in \Z$ and all $q \in M$, it holds that $q \ll \phi^J(q)$.
 \end{enumerate}
 If (ii) holds, then the action of $\phi$ is free, proper, and cocompact. Moreover, if (ii) holds for $J>0$, then it cannot simultaneously hold for $J<0$.
\end{theorem}

As in the introduction, we call the conformal transformations satisfying (i) \emph{non-escaping}, and the ones satisfying (ii) \emph{escaping}. Furthermore, we distinguish between \emph{future-escaping} if $J>0$ and \emph{past-escaping} if $J<0$. Note that condition (i) is equivalent to requiring that every compact set $K$ is contained in some $\phi$-invariant compact set $K'$. This is simply because $\bigcup_{j \in \Z} \phi^j(K)$ is manifestly $\phi$-invariant, so its closure must be an invariant compact set.

\begin{proof}
 By Theorem~\ref{thm:Haw-Mal}, $\phi$ preserves the chronological relation $\ll$. Let $\tau \colon M \to \R$ be a Cauchy temporal function, so that $M$ is foliated by its level sets $\Sigma_t := \tau^{-1}(t)$, and $M = \R \times \Sigma_0$ splits as in Theorem~\ref{thm:GerBeSa}. Suppose that $K \subset M$ is a compact set that does not satisfy (i). We first show that then the action is free, proper, and cocompact.

 Because $K$ is compact, there is $a>0$ such that $K \subset [-a,a] \times \Sigma_0$. By the NOH and Lemma~\ref{lem:NOHCauchy}, we may choose $a>0$ large enough, and find a suitable $b > a$, such that $\Sigma_{-b}, \Sigma_{-a}, \Sigma_0, \Sigma_a, \Sigma_b$ are pairwise totally timelike connected, in this order. By assumption, there is a point $p \in K$ whose orbit leaves $[-b,b] \times \Sigma_0$. Assume, without loss of generality, that $p \in \Sigma_0$ and that there is a $N \in \Z$ such that $\tau(\phi^N(p))>b$. We denote $\psi := \phi^N$.

 Because $\psi$ preserves the causal structure, $\psi(\Sigma_0)$ is a Cauchy surface. From $\tau(\psi(p))>b$ and the choice of $a$ and $b$, it follows that $\Sigma_a \subset I^-(\psi(p)) \subset I^-(\psi(\Sigma_0))$. Hence $\psi(\Sigma_0) \subset I^+(\Sigma_a) \subset I^+(\Sigma_0)$, and $\Sigma_0$ and $\psi(\Sigma_0)$ are totally timelike connected. We claim that
 \begin{equation*}
  M_0 := J^+(\Sigma_0) \cap I^-(\psi(\Sigma_0))
 \end{equation*}
 is a fundamental domain for the action of $\psi$.

 Suppose $q \in M_0$. Then, for all $k >0$,
 \begin{align*}
  \psi^{k}(q) \in J^+(\psi^{k}(\Sigma_0)) &\implies \psi^{k}(q) \notin I^-(\psi^{k}(\Sigma_0)) \supset I^-(\psi(\Sigma_0)) \\ &\implies \psi^{k}(q) \notin M_0.
 \end{align*}
 Similarly, for $k <0$,
 \begin{align*}
  \psi^{k}(q) \in I^-(\psi^{k+1}(\Sigma_0)) &\implies \psi^{k}(q) \notin J^+(\psi^{k+1} (\Sigma_0)) \supset J^+(\Sigma_0) \\ &\implies \psi^{k}(q) \notin M_0.
 \end{align*}
 Thus every orbit has at most one representative in $M_0$. Moreover, this argument also shows that no point in $M_0$ is fixed.

 It remains to show that every orbit has at least one representative in $M_0$ (which together with the above then also implies that $\psi$ has no fixed points). This is equivalent to showing that
 \begin{equation} \label{eq:union}
  M = \bigcup_{i \in \Z} M_i, \quad \text{where} \quad M_i := \psi^i(M_0) = J^+(\psi^i(\Sigma_0)) \cap I^-(\psi^{i+1}(\Sigma_0)).
 \end{equation}
 Note that $M_{i-1} \subset I^+(M_{i})$, that $M_{i-1}$ and $M_{i}$ are directly adjacent, and that the boundary between the two is precisely $\psi^i(\Sigma_0)$. Moreover, $\psi^i(\Sigma_0)$, being a Cauchy surface in $M= \R \times \Sigma_0$, is given by the graph of a function $F_i \colon \Sigma_0 \to \R$. Furthermore, since $\psi^i(\Sigma_0) \subset I^+(\psi^{i-1}(\Sigma_0))$, it holds that $F_i > F_{i-1}$. Thus \eqref{eq:union} holds if $F_i(x) \to \pm\infty$ as $i \to \pm\infty$, for every $x \in \Sigma_0$.

 Suppose that there is a $x \in \Sigma_0$ such that $F_i(x)$ is bounded above (the case of being bounded below can be treated analogously). Then $F_i(x) \to C < \infty$ as $i \to \infty$. By Lemma~\ref{lem:NOHCauchy}, there is a value $C' > C$ such that $\Sigma_C$ and $\Sigma_{C'}$ are totally timelike connected. It follows that $F_i(y)<C'$ for all $y \in \Sigma_0$, because the level sets of each $F_i$, being Cauchy surfaces, must be achronal. Hence, in fact, the sequence of functions $F_i$ converges pointwise to some limit function $F$. Let $x \neq y \in \Sigma_0$. Moreover, $\Sigma_0$ and $\psi(\Sigma_0)$ being totally timelike connected implies that so are $F_{i}$ and $F_{i+1}$, for all $i$. Thus $(F_{i-1}(y),y) \ll (F_i(x),x) \ll (F_{i+1}(y),y)$. Taking the limit $i \to \infty$, and using that the causal relation $\leq$ is closed (by global hyperbolicity), we conclude that $(F(y),y) \leq (F(x),x) \leq (F(y),y)$. This contradicts causality, and concludes the proof of \eqref{eq:union}.

 Having proven that $M_0$ is a fundamental domain, using \cite[Lem.~21.11]{Lee} it is easy to see that the action of $\psi := \phi^N$ is proper. Moreover, since $M_0$ has compact closure, it follows that the action is cocompact. Thus also the action of $\phi$ must be free, proper, and cocompact (see Lemma~\ref{lem:powers}).

 Now, by total timelike connectedness of the boundaries between the $M_i$'s, it is easy to see that for $J=2N$, $\phi^{J} = \psi^2$ has the property $q \ll \psi^2(q)$ for all $q \in M$. Let us prove that there cannot simultaneously exist $J > 0$ and  $J' > 0$ such that $q \ll \phi^J(q)$ and $q \gg \phi^{J'}(q)$ for all $q \in M$. Suppose that such $J,J'$ did exist. Then $q \ll \phi^J(q) \ll \phi^{2J}(q) \ll ... \ll \phi^{J' J}(q)$. But analogously also $q \gg \phi^{J J'}(q)$, a contradiction. This proves the distinction between future- and past-escaping transformations.

 Finally, we prove that statements (i) and (ii) are mutually exclusive. Suppose that both hold. By (i), $p \ll \phi^J (p)$ for all $p \in M$. Then, the sequence $p_k := \phi^{kJ} (p)$ forms a chronological chain, i.e.\ $p_k \ll p_{k+1}$, and we can find a timelike curve $\gamma$ with $\gamma(k) = p_k$. By global hyperbolicity, there are two possibilities:
 \begin{enumerate}
  \item $\gamma$ has a future endpoint $p_\infty$, and $p_k \to p_\infty$. Necessarily, $p_\infty$ is a fixed point of $\phi^J$.
  \item $\gamma$ is future inextendible, $p_k$ does not converge, and $\phi$ is escaping.
 \end{enumerate}
 We need to show that the first case cannot occur. Suppose it did. Because $\phi^J$ preserves the causal relation, as well as the chronological one, the set
 \begin{equation*}
  H^+(p_\infty) := J^+(p_\infty) \setminus I^+(p_\infty)
 \end{equation*}
 must be preserved by $\phi^J$. However, $H^+(p_\infty)$ is non-empty and achronal, while $q \ll \phi^J(q)$ for all $q \in M$, a contradiction.
\end{proof}

We end this section by proving some additional facts about the set $\mathcal{E}^+ \subset \Conf^\uparrow(M,g)$ of future-escaping conformal transformations.
\newpage

\begin{proposition} \label{prop:conj}
 Let $(M,g)$ be a causal spacetime satisfying the NOH. Then:
 \begin{enumerate}
  \item $\mathcal{E}^+$ is open and invariant under conjugation.
  \item If $t \mapsto \Psi_t$ is a one-parameter subgroup with $\Psi_1 \in \mathcal{E}^+$, then $\Psi_t \in \mathcal{E}^+$ for all $t > 0$.
 \end{enumerate}
\end{proposition}

\begin{proof}
 (i) That $\mathcal{E}^+$ is invariant under conjugation is natural, but we provide a short proof. Suppose $\phi = \chi \circ \psi \circ \chi^{-1}$ for $\chi,\psi \in \Conf^\uparrow(M,g)$. If $\psi \in \mathcal{E}^+$, then there is a $J \in \N$ such that $p \ll \psi^J (p)$ for all $p \in M$. Since $\chi^{-1} \circ \phi^J = \psi^J \circ \chi^{-1}$, it follows that $\chi^{-1} (p) \ll \chi^{-1} \circ \phi^J (p)$. Because $\chi$ is conformal, it follows that $p \ll \phi^{J}(p)$.

 That $\mathcal{E}^+$ is open follows from careful analysis of the proof of Theorem~\ref{thm:escap}. There we showed that for each $\phi \in \mathcal{E}^+$, there is some power $\psi = \phi^N$ and a Cauchy surface $\Sigma_0$ such that $\Sigma_0$ and $\psi(\Sigma_0)$ are totally timelike connected. Since the relation $\ll$ is open and $\Sigma_0$ is compact, there is a neighborhood $U \subset M$ of $\psi(\Sigma_0)$ such that $\Sigma_0$ and $U$ are totally timelike connected. Therefore, there is a neighborhood $\mathcal{U} \subset \Conf^\uparrow(M,g)$ of $\psi$ such that $\Sigma_0$ and $\chi(\Sigma_0) \subset U$ are totally timelike connected, for every $\chi \in \mathcal{U}$ (hence $\chi$ is future escaping). Moreover, there is a neighborhood $\mathcal{V} \subset \Conf^\uparrow(M,g)$ of $\phi$ such that $\theta^N \in \mathcal{U}$ for all $\theta \in \mathcal{V}$. It follows that $\mathcal{V} \subset \mathcal{E}^+$, so $\mathcal{E}^+$ is open.

 (ii) If $\Psi_1$ is future-escaping, then so are all its positive rational powers $\Psi_1^{k/l} = \Psi_{k/l}$. These form a dense subset of $\Psi_{(0,\infty)}$, and since being future-escaping is an open condition , it follows that $\Psi_{(0,\infty)} \subset \mathcal{E}^+$.
\end{proof}

We extract the following consequences from Proposition~\ref{prop:conj}:
\begin{itemize}
 \item Also the set of past-escaping elements is open and invariant under conjugation. Therefore, the set of non-escaping transformations is closed and invariant under conjugation.
 \item If non-escaping elements form a subgroup, then it is a normal Lie subgroup. However, Example~\ref{ex:notsubgroup} below shows that non-escaping elements do not always form a subgroup. (Escaping elements clearly do not form a subgroup.)
 \item It follows from (ii) that the notions of escaping and non-escaping are well-defined on the Lie algebra of $\Conf^\uparrow(M,g)$. Future-, past-, and non-escaping elements each form a cone in the Lie algebra.
\end{itemize}
An interesting question, which lies beyond the scope of this article, is to characterize the cones of (non)-escaping elements in the Lie algebra. This is non-trivial even in concrete examples (such as the Einstein static universe appearing in the next section). Cone structures on Lie algebras, and the ensuing ``causal'' orders on Lie groups and their symmetric spaces, have been extensively studied (see e.g.\ the book \cite{HiOl}).


\section{The Einstein static universe}

Here, we apply the results of the previous section to the Einstein static universe. In particular, we study how the dichotomy between escaping and non-escaping elements compares to the Jordan decomposition of a matrix into elliptic, hyperbolic and parabolic factors. The Einstein static universe is perhaps the simplest example of causal spacetime satisfying the NOH. But precisely because it is so symmetric, it has a very large conformal group, so in that sense, it is the richest example. We do not claim novelty of all lemmas in this section, which we prove to keep the paper self-contained. Our main new contribution is Theorem~\ref{thm:classification} below.

\subsection{Description of the Einstein universes}

Let $\R^{2,n+1}$ be the pseudo-Euclidean space of signature $(2,n+1)$, and let $\Lambda \subset \R^{2,n+1}$ be the isotropic (or null) cone, i.e.\ the set of vectors with vanishing pseudo-norm. By (compact) \emph{Einstein universe}, denoted $\Eins^{1,n}$, we mean the quotient of $\Lambda \setminus  \{0\}$ by the relation
\begin{equation} \label{eq:equivrel}
 v \sim w \colon \iff \exists c \in (0,\infty) \text{ such that } v = cw.
\end{equation}
The metric of the ambient $\R^{2,n+1}$ defines a class of conformally equivalent Lorentzian metrics on $\Eins^{1,n}$ (below, we choose a canonical representative). The action of $\OG(2,n+1)$ on $\R^{2,n+1}$ thus descends to an action on $\Eins^{1,n}$ by conformal transformations. It is well-known that, in fact, $\Conf(\Eins^{1,n}) = \OG(2,n+1)$. Note that many authors call Einstein universe the quotient \eqref{eq:equivrel} but with $c \in \R \setminus \{0\}$, the space that we define being a double cover thereof. We shall also use the notation $\Conf^\uparrow(\Eins^{1,n}) = \OG^\uparrow(2,n+1)$.

Let us explicitly construct the quotient. Denote points in $\R^{2,n+1}$ by $(x,y,z)$, where $x,y \in \R$ and $z \in \R^{n+1}$. The condition $(x,y,z) \in \Lambda$ corresponds to
\begin{equation*}
 -x^2-y^2+\Vert z \Vert^2 = 0,
\end{equation*}
and one can fix a representative for each equivalence class of \eqref{eq:equivrel} by imposing
\begin{equation*}
 x^2+y^2+\Vert z \Vert^2 = 2.
\end{equation*}
Combining the above implies that $(x,y) \in \bS^1$ and $z \in \bS^n$, not only as sets, but also in the metric sense (but with negative definite metric on $\bS^1$). We get from this a canonical choice of metric in the conformal class of $\Eins^{1,n}$, namely
\begin{equation*}
 g := -d\theta^2 + g_{\bS^n},
\end{equation*}
where $\theta$ is the coordinate on $\bS^1$ and $g_{\bS^n}$ is the round metric on the $n$-sphere.

The \emph{Einstein static universe} $\hEins^{1,n}$ is the cyclic cover of $\Eins^{1,n}$ obtained by covering the $\bS^1$ factor by $\R$, therefore it is isometric to $\R \times \bS^n$ with metric
\begin{equation*}
 \hat{g} = -dt^2 + g_{\bS^n}.
\end{equation*}
It coincides with the universal cover of $\Eins^{1,n}$ when $n \geq 2$. As a spatially compact product spacetime, $\hEins^{1,n}$ is globally hyperbolic. It is known (see e.g.\ \cite[Prop.~3.6]{GaZe} for a proof) that the group of time-orientation preserving isometries is
\begin{equation*}
 \Isom^\uparrow(\hEins^{1,n}) = \R \times \OG(n+1),
\end{equation*}
where $\R$ acts by time translations, and $\OG(n+1)$ by symmetries of the $n$-sphere $\bS^n$.

The covering map $\pi \colon \hEins^{1,n} \to \Eins^{1,n}$ induces a map of groups
\begin{equation*}
 \pr \colon \R \times \OG(n+1) \longrightarrow \SO(2) \times \OG(n+1),
\end{equation*}
which preserves $\OG(n+1)$ and restricts to the usual covering $\R \to \SO(2)$. Note, additionally, that $\SO(2) \times \OG(n+1)$ is a maximal compact subgroup of $\OG^\uparrow(2,n+1)$ \cite[App.~C]{Knapp}.

   \subsection{Structure of $\Conf(\Eins^{1,n})$ and $\Conf(\hEins^{1,n})$}

   \begin{proposition} \label{prop:cover}
    The group $\Conf(\Eins^{1,n})$ equals $\OG(2,n+1)$. The group $\Conf(\hEins^{1,n})$ is a central extension thereof, denoted by $\hOG(2,n+1)$, and given by the short exact sequence
    \begin{equation*}
     1 \longrightarrow 2 \pi \Z \overset{\iota} \longrightarrow \hOG(2,n+1) \overset{\pr} \longrightarrow \OG(2,n+1) \longrightarrow 1.
    \end{equation*}
   Here $\iota(2\pi)$ is the conformal transformation of  $\hEins^{1,n} = \R \times \bS^n$ that maps $(t,x) \mapsto (t+2\pi,x)$.
   \end{proposition}

   \begin{proof}
    This is well-known, see e.g.\ \cite{EMN}. Note that by construction of $\Eins^{1,n}$, it is easy to see that $\OG(2,n+1)$ acts by conformal transformations on $\Eins^{1,n}$, since the action of $\OG(2,n+1)$ on $\R^{2,n+1}$ preserves the isotropic cone. For illustrative purposes, we prove the second part. We view the compact Einstein universe $\Eins^{1,n} \cong \bS^1 \times \bS^n$ as the quotient of the Einstein static universe $\hEins^{1,n} \cong \R \times \bS^n$ by the isometry $T_{2\pi} \colon (t,x) \mapsto (t+2\pi,x)$. Because $T_{2\pi}$ is in the center of $\Conf(\hEins^{1,n})$, it is easy to see that every conformal transformation of $\hEins^{1,n}$ descends to the quotient. This defines the group homomorphism $\pr$. Suppose that $\phi \in \Conf(\hEins^{1,n})$ is such that $\pr(\phi) = 1$. Then, $\phi$ must act by $(t,x) \mapsto (t+2\pi k(t,x),x)$, where $k(t,x) \in \Z$. By continuity and connectedness, $k(t,x)$ must be constant, and $\phi = (T_{2\pi})^k$. This proves that $\iota(2\pi\Z) = \ker(\pr)$, so the above sequence is exact.
   \end{proof}

  Recall the ``algebraic'' definitions for matrices $A \in \GL(n,\R)$:

  \begin{itemize}
   \item $A$ is hyperbolic if it is diagonalizable with all eigenvalues real and positive.
   \item $A$ is parabolic if it is unipotent, i.e.\ if $A - 1$ is nilpotent.
   \item $A$ is elliptic if it is diagonalizable over $\mathbb{C}$ and all its eigenvalues have modulus $1$.
  \end{itemize}

   \begin{proposition}[Jordan decomposition in $\OG(2,n+1)$] \label{prop:JordanO}
    Every element $A \in \OG(2,n+1)$ can be written as a commuting product $A = E H P$, with $E$ elliptic, $H$ hyperbolic, and $P$ parabolic. Here $E,H,P \in \OG(2,n+1)$.
   \end{proposition}

   This follows from a general fact about algebraic groups \cite[Chap.~IV]{Hum}, but here we give a simpler argument for the case at hand.

   \begin{proof}
    A proof of the Jordan decomposition in the general linear group $\GL(n+3,\R)$ can be found in \cite[Sec.~IX.7]{Hel}. It follows from the proof there (including \cite[Sec.~III.1]{Hel} and the fact that $A^{-1}$ is a polynomial in $A$) that $E, H, P$ can be expressed as polynomials in $A$. We use this to prove that $E,H,P \in \OG(2,n+1)$ (as opposed to just in $\GL(n+3,\R)$). We show this more generally for $\OG(p,q)$ and, in the rest of this proof, we denote $n := p+q$.

    Let $Q$ be the space of quadratic forms on $\R^{n}$. We have a representation $\rho: \GL(n, \R) \to \GL(Q)$. Let $b$ be the quadratic form $-(x_1^2 + \ldots  + x_p^2) + (y_1^2 + \ldots + y_q^2)$. Therefore
    \begin{equation*}
     \OG(p, q) = \{A \in \GL(n, \R) \mid \rho(A)b = b\}.
    \end{equation*}
    Now suppose $A \in \times \OG(p, q) $, and $E$ is its elliptic part. Then $E$ is a polynomial in $A$, and thus $\rho(E)$ is a polynomial in $\rho(A)$. Therefore $\rho(E)$ preserves the line $\R b \subset Q$, and the same holds for $\rho(P)$ and $\rho(H)$. But $E$ being elliptic implies that, in fact, it must preserve $b$, and hence $E \in \OG(p,q)$. Similarly, $P$ being parabolic means that it has at least one eigenvector with eigenvalue $1$, so $P$ must preserve $b$ as well. Finally, this implies that also $H$ must preserve $b$, and we are done.
   \end{proof}

   We can now extend the definitions of elliptic, hyperbolic, and parabolic elements from $\OG(2,n+1)$ to $\hOG(2,n+1)$.

   \begin{definition}
    Let $\pr \colon \hOG(2,n+1) \to \OG(2,n+1)$ be as in Proposition \ref{prop:cover}. An element $\phi \in \hOG(2,n+1)$ is called
    \begin{itemize}
     \item \emph{Elliptic} if $\pr(\phi)$ is elliptic,
     \item \emph{Hyperbolic} if $\phi$ has a fixed point and $\pr(\phi)$ is hyperbolic,
     \item \emph{Parabolic} if $\phi$ has a fixed point and $\pr(\phi)$ is parabolic.
    \end{itemize}
   \end{definition}

   The motivation for requiring hyperbolic and parabolic elements to have fixed points is that we want to view non-trivial time translations $(t,x) \mapsto (t+c,x)$ as elliptic, but not hyperbolic or parabolic (even when $c=2\pi k$, in which case $\pr(\phi) = 1$). We will see in Section~\ref{sec:hyperpara} below that hyperbolic and parabolic elements in $\OG(2,n+1)$ have fixed points in $\Eins^{1,n}$, so it makes sense to require this.



   \subsection{Elliptic elements}

   \begin{lemma} \label{lem:ell}
    An element $E \in \OG(2,n+1)$ is elliptic if and only if $E$ is conjugate to an element of the subgroup $\OG(2) \times \OG(n+1)$.
   \end{lemma}

   \begin{proof}
    Suppose $E$ is elliptic. If we view $E$ as an element in $\Mat(n+3, \C)$, then there is $B$ invertible and $D$ diagonal such that $E = B^{-1} D B$, and by definition of ellipticity, all diagonal entries of $D$ have modulus $1$. Hence, for the inner product on $\C^{n+3}$ represented by $B^T B$, it holds that $E$ and all its powers $E^k = B^{-1} D^k B$ have operator norm equal to $1$.

    Let $\Gamma = \{ E^k \mid k \in \Z \}$. By the previous discussion, its closure $\overline{\Gamma}$ is a compact subset of $\Mat(n+3, \C)$. But, in fact, $\overline{\Gamma}$ is a subgroup of $\OG(2,n+1)$. Because $\OG(2) \times \OG(n+1)$ is a maximal compact subgroup of $\OG(2,n+1)$, it follows that $\overline{\Gamma}$ is, up to conjugation, a subgroup of $\OG(2) \times \OG(n+1)$ (see \cite[Sec.~IV.2]{Hel} and \cite[App.~C]{Knapp}). This proves one direction.

    The converse follows from the standard fact that all elements of the orthogonal group $\OG(m)$ are elliptic.
   \end{proof}

    \begin{lemma} \label{lem:ellincover}
    Let $\phi \in \hOG^\uparrow(2,n+1)$ be elliptic and time-orientation preserving. The following are equivalent:
    \begin{enumerate}
     \item $\phi$ is non-escaping,
     \item $\pr (\phi)$ is conjugate to an element $E \in \OG(n+1)$,
     \item $\phi$ preserves a time function.
    \end{enumerate}
    \end{lemma}

    \begin{proof}
     (i) $\iff$ (ii). By Lemma~\ref{lem:ell} above, $\phi$ being elliptic implies that
    \begin{equation*}
     \pr \phi = A E A^{-1}
    \end{equation*}
    for $A \in \OG(2,n+1)$ and $E = (E_1,E_2) \in \OG(2) \times \OG(n+1)$. Let $\psi \in \pr^{-1}(A)$ and $\phi_E := \psi \circ \phi \circ \psi^{-1}$. Then $\pr (\phi_E) = E$, and by Proposition~ \ref{prop:conj}, $\phi_E$ is non-escaping. This is only possible if $E_1 = 1$, since otherwise $E_1$ would lift to either a time translation in $\hOG(2,n+1)$ (in which case $\phi_E$ would be escaping), or to a time-orientation reversing transformation (which we have excluded by assumption).

    (ii) $\iff$ (iii). As discussed above, $\phi = \psi^{-1} \circ \phi_E \circ \psi$, where $\pr(\phi_E) \in \OG(n+1)$. Hence $\phi_E$ preserves the canonical time coordinate $t$ on $\hEins^{1,n} = \R \times \bS^n$, and therefore $\phi$ preserves the time function $t \circ \psi$ (note that being a time function is preserved by composition with a  time-orientation preserving conformal transformation).
    \end{proof}

    \begin{remark} \label{rem:ellproper}
     If $\phi \in \hOG^\uparrow(2,n+1)$ is elliptic, then the closure of the subgroup generated by $\phi$ acts properly. This is because escaping elements always act properly, while every non-escaping elliptic element is part of some compact subgroup. It follows by \cite[Thm.~3.2]{GaZe} that if $\phi$ is escaping, then $\phi$ preserves the differential of a temporal function, and hence the foliation by its level sets (but not each individual level set).
    \end{remark}

   \subsection{Hyperbolic and parabolic elements} \label{sec:hyperpara}

   \begin{lemma} \label{lem:hyper}
    Let $H \in \OG(2,n+1)$ be hyperbolic. Then, there is an eigenbasis of $\R^{2, n+1}$, where the pseudo-scalar product has the form
    \begin{equation*}
     2xy + 2 zt + w_1^2+ \ldots + w_{n-1}^2,
    \end{equation*}
    and $H$ has eigenvalues $(\lambda_1, \lambda_1^{-1}, \lambda_2, \lambda_2^{-1}, 1, \ldots 1)$.
   \end{lemma}

   \begin{proof}
    If all eigenvalues are $1$, then $H=1$, and the statement reduces to the well-known fact that the pseudo-scalar product has the desired form in some appropriate basis. Suppose thus that there is an eigenvector $e$ with positive eigenvalue $\lambda \neq 1$. Necessarily $e$ is isotropic (otherwise the pseudo-norm of $e$ would not be preserved by $H$). The orthogonal complement $e^\perp$ is a degenerate hyperplane. Since $H$ is diagonalizable, there is an $H$-invariant subspace $F$ such that $e^\perp = \R e \oplus F$. The restriction of the pseudo-scalar product on $F$ has Lorentzian signature, and is preserved by $H$. The orthogonal complement $F^\perp$ is a Lorentztian $2$-plane spanned by $e$ and by another vector $\tilde{e}$, which can be chosen so that the pseudo-scalar product on $F^\perp$ is given by $2 e^* \tilde{e}^*$. Because this pseudo-scalar product is preserved by $H$, necessarily $\tilde{e}^*$ must be an eigenvector with eigenvalue $\lambda^{-1}$. It remains to analyse the restriction of $H$ on $F$. This is completely analogous: either the restriction is trivial, or it has an isotropic eigenvector with eigenvalue $\lambda_2 \neq 0$, and we can repeat the same argument. Due to the signature of the pseudo-scalar product, there can be at most four eigenvalues different from $1$. At that point, the orthogonal complement to the corresponding (isotropic) eigenvectors is positive-definite, and hence $H$ must act trivially on it (because $H$ is hyperbolic).
   \end{proof}

   \begin{lemma}\label{lem:para}
    If $P \in \OG(2,n+1)$ is parabolic, then $P$ preserves some isotropic vector (i.e.\ there is an isotropic eigenvector with eigenvalue $1$). In particular, $P$ has a fixed point in $\Eins^{1,n}$.
   \end{lemma}

   \begin{proof}
    Because $P$ is parabolic, there is some non-zero vector $v \in \ker(P - {1})$. If $v$ is isotropic, we are done. If not, then $\R^{2,n} = \R v \oplus v^\perp$, and each factor is preserved by $P$. Moreover, the restriction of $P$ to each factor is again parabolic. Hence, we can repeat the same argument on $v^\perp$. After finitely many steps, we obtain an invariant vector $w \in v^\perp$ that is either isotropic, or such that the squared pseudo-norms $\Vert v \Vert^2$ and $\Vert w \Vert^2$ have opposite signs. In the latter case, there is some linear combination of $v$ and $w$ that is isotropic and invariant.
   \end{proof}

   \subsection{Mixed elements}

   \begin{lemma}
    If $A \in \OG(2,n+1)$ has a Jordan decomposition $A = EPH$ with $E = 1$, then $A$ has a fixed point in $\Eins^{1,n}$.
   \end{lemma}

   \begin{proof}
   This means $A = PH$, with $P$ and $H$ commuting. Then $P$ must preserve the 4-plane $\R^4$ generated by the first coordinates $x, y, z, t$ as in Lemma~\ref{lem:hyper}. Since $P$ is parabolic, it has a non-trivial subspace $L \subset \R^4$ on which it acts trivially.  Since $P$ and $H$ commute, then $L$ is $H$-invariant, and hence contains an eigendirection of $H$, which is thus an isotropic eigendirection of $PH$, and a fixed point in $\Eins^{1, n}$.
   \end{proof}

   \begin{lemma} \label{lem:PHnot1}
    If $A \in \OG(2,n+1)$ has a Jordan decomposition $A =EPH$ with $H \neq 1$ and $P \neq 1$, then $A$ has a fixed point in $\Eins^{1,n}$.
   \end{lemma}

   \begin{proof}  Consider the splitting $\R^{2,2} \oplus \R^{n-1}$ obtained in Lemma~\ref{lem:hyper}, where $\R^{2,2}$ is spanned by four isotropic eigenvectors of $H$, denoted by $e_1,...e_4$, with eigenvalues $\lambda_1,\lambda_1^{-1},\lambda_2,\lambda_2^{-1}$. Since $H \neq 1$, we have that $\lambda_1 \neq 0$.
   \begin{itemize}
    \item If $H$ has different eigenvalues, that is $\lambda_1 \neq \lambda_2$, then each matrix commuting with $H$ preserves the isotropic eigendirection $\R e_1$, and we are done.

    \item If $\lambda_1 = \lambda_2$, then $P$ and $E$ preserve the two isotropic 2-planes $\R e_1 \oplus \R e_3$ and $\R e_2 \oplus \R e_4$. On each, $P$ is parabolic, and since $P$ is non-trivial, it acts non-trivially on one of them, say $\R e_1 \oplus \R e_3$. (Note that if $P$ was trivial on $\R^{2,2}$, then it would be an element of the compact subgroup $\OG(n-1)$, hence elliptic.) Now, $E$ acts as a rotation on $\R e_1 \oplus \R e_3$, and commutes with $P$. Hence $E$ is trivial, because it must preserve the one-dimensional subspace $\ker(P - 1)$. By Lemma~\ref{lem:para}, $P$ preserves some isotropic vector in $\R e_1 \oplus \R e_3$, giving us the desired isotropic eigendirection for $A$.
    \end{itemize}
    \end{proof}

    \begin{lemma} \label{lem:plane}
     If $A \in \OG(2,n+1)$ has a Jordan decomposition $A =EPH$ with $H \neq 1$, $P = 1$ or with $H = 1$, $P \neq 1$, then $A$ preserves an isotropic 2-plane, or an isotropic direction, in $\R^{2,n+1}$.
    \end{lemma}

    In the case that an isotropic direction is preserved, we again get a fixed point in $\Eins^{1,n}$.

     \begin{proof}
      If $H \neq 1$ and $\lambda_1 \neq \lambda_2$ then, because $E$ and $H$ commute, $E$ preserves the isotropic direction $\R e_1$. Similarly, if $H \neq 1$ and $\lambda_1 = \lambda_2$, then $E$ preserves the isotropic 2-planes $\R e_1 \oplus \R e_3$ and $\R e_2 \oplus \R e_4$.

      If $H = 1$ but $P \neq 1$, consider the space $I$ generated by isotropic fixed vectors of $P$. By Lemma \ref{lem:para}, $I$ is non-trivial, and $P$ acts trivially on it. Because $P$ and $E$ commute, $E$ preserves $I$ as well. Moreover, $P$ and $E$ also preserve $I^\perp$, and the restriction of $P$ to $I^\perp$ is again parabolic. If $I$ is non-degenerate, then $I \cap I^\perp = \{ 0\}$ and $I^\perp$ is also non-degenerate. If the signature on $I^\perp$ is indefinite, then by Lemma~\ref{lem:para}, there is an isotropic fixed vector in $I^\perp$, a contradiction. If the signature on $I^\perp$ is definite, one can keep constructing invariant vectors as in the proof of Lemma~\ref{lem:para}, concluding that $P$ acts trivially on $I^\perp$. But then $P = 1$, a contradiction. We conclude that $I$ is degenerate, and hence $I^\prime = I \cap I^\perp$ is a non-trivial, isotropic subspace invariant under $P$, $E$ and $A$. Note that isotropic subspaces of $\R^{2,n+1}$ can only be of dimension $1$ or $2$, so we are done.
     \end{proof}

  \subsection{Conclusion}

  The preceding discussion allows us to classify conformal transformations of the Einstein static universe. The following two Theorems, together with Lemma~\ref{lem:ellincover}, imply Theorem~\ref{thm:intro2} from the introduction.

  \begin{theorem} \label{thm:classification}
  Let $\phi \in \hOG^\uparrow(2,n+1)$ act on $\hEins^{1,n}$ by a conformal transformation. Then, there are three possibilities:
  \begin{enumerate}
   \item $\phi$ is non-escaping elliptic and preserves a time function.
   \item $\phi$ is non-escaping and has a fixed point.
   \item $\phi$ is escaping and is a commuting product of an elliptic escaping element and a (possibly trivial) transformation with a fixed point.
  \end{enumerate}
  \end{theorem}

  Note that isometries of $\bS^n$ can have fixed points, in which case the first- and second conditions are satisfied simultaneoulsy.

 \begin{proof}
  Let $A := \pr \phi = EPH \in \OG(2,n+1)$. If $\phi$ is non-escaping and $P=H=1$, then by Lemma~\ref{lem:ell}, $\phi$ preserves a time function. Next, suppose that $A$ has a fixed point $p \in \Eins^{1,n}$. Let $(t,x) \in \pi^{-1}(p)$, where $\pi \colon \hEins^{1,n} \to \Eins^{1,n}$ is the quotient map. Then $\phi(t,x) = (t+2\pi k,x)$ for some $k \in \Z$. If $k=0$, then $\phi$ has a fixed point, and if $k \neq 0$, then $\phi$ is escaping.

  It remains to deal with the case that $A$ is not elliptic and does not have a fixed point. Then, by Lemma~\ref{lem:PHnot1}, $P$ or $H$ must be trivial, and by Lemma~\ref{lem:plane}, $E$ preserves an isotropic 2-plane $\mathcal{E}$ on which it acts as a non-trivial rotation. The projection of $\mathcal{E}$ is a null geodesic in $\Eins^{1, n}$, which lifts to a null geodesic $\gamma$ in $\hEins^{1,n}$, necessarily of the form $\gamma(s) = (t(s),x(s))$ with $\dot{t}(s) = \Vert \dot{x}(s)\Vert_{\bS^n} > 0$. Now, up to conjugacy, we can assume $E \in \OG(2) \times \OG(n+1)$. A non-trivial element $E \in \OG(1+n)$ cannot preserve $\gamma$, since necessarily $(t(s),x(s)) \neq (t(s),Ex(s))$ for some $s$. This implies that $E$ has a non-trivial projection onto $\OG(2)$, which must lie in $\SO(2)$, because we assumed that $\phi$ is time-orientation preserving. Hence $E$ is escaping. Moreover, $P$ and $H$ each have a fixed point (see Lemmas \ref{lem:hyper} \& \ref{lem:para}), and one of them is trivial. Thus the third condition of the theorem is satisfied.
 \end{proof}

  \begin{theorem} \label{thm:conjugacy}
   Let $\phi, \psi \in \hOG^0(2,n+1)$ be future-escaping and contained in the connected component of the identity. Then there exists a diffeomorphism $\chi \colon \Eins^{1,n} \to \Eins^{1,n}$ (not necessarily conformal) such that $\psi = \chi \circ \phi \circ \chi^{-1}$.
  \end{theorem}

  \begin{proof}
   The proof is divided in a series of steps.

   \textbf{Claim 1.} \textit{There is a smooth path $[0,1] \to \hOG^0(2,n+1),\ t \mapsto \phi_t$ such that $\phi_0 = \phi$, $\phi_1$ is elliptic, and $\phi_t$ is future-escaping for all $t \in [0,1]$.}

  By Theorem \ref{thm:classification}, $\phi = f \circ \phi_E$, where $f$ has a fixed point $p_0$, and $\phi_E$ is elliptic escaping. By Proposition~\ref{prop:fixed}, $f$ corresponds canonically to some element $\bar{f} \in \Conf(\R^{1,n})$. Moreover, because the connected components of the identity in $\Conf(\R^{1,n})$ and $\hOG(2,n+1)$ are characterized by preserving orientation and time-orientation, we have that $\bar{f} \in \Conf^0(\R^{1,n})$. Thus we can find a path $t \mapsto \bar{f}_t$ such that $\bar{f}_0 = \bar{f}$ and $\bar{f}_1 = \Id$ is the identity. But this then canonically defines a path $t \mapsto f_t$ in $\hOG^0(2,n+1)$ such that $f_0 = f$, $f_1 = \Id$, and $f_t$ fixes the same point $p_0$ for all $t \in [0,1]$. Then
  \begin{equation*}
   t \mapsto \phi_t := f_t \circ \phi_E
  \end{equation*}
  is the path we are looking for. In particular, $\phi_t$ is escaping for all $t$, because $\phi_t^k(p_0) = \phi_E^k(p_0)$ must leave every compact set as $k \to \infty$.

  \textbf{Claim 2.} \textit{The action of \[\Phi \colon [0,1] \times \hEins^{1,n} \longrightarrow [0,1] \times \Eins^{1,n},\quad (t,x) \longmapsto (t,\phi_t(x))\] is free and proper.}

  Freeness follows easily because for each $t \in [0,1]$, $\phi_t$ has no fixed points. To show properness, let $A \subset [0,1] \times \Eins^{1,n}$ be compact. Then $A \subset \tilde{K} := [0,1] \times K$ for some compact $K \subset \Eins^{1,n}$. We need to show that
  \begin{equation*}
   G_{\tilde{K}} = \{j \in \Z \mid \Phi^j(\tilde{K}) \cap \tilde{K} \neq \emptyset \}
  \end{equation*}
  is finite. Note that
  \begin{equation*}
   \Phi^j(\tilde{K}) \cap \tilde{K} = \bigcup_{t \in [0,1]} \{t\} \times \left( \phi_t^j(K) \cap K \right).
  \end{equation*}
  For each $t \in [0,1]$, $\phi_t$ is escaping and hence acts properly. Thus, there is $J_t \in \Z$ such that
  \begin{equation*}
   \phi_t^j(K) \cap K = \emptyset \quad \text{for all } j \geq J_t.
  \end{equation*}
  Because $t \mapsto \phi_t$ is continuous, there is a $\delta > 0$ small enough such that
  \begin{equation*}
   \phi_s^j(K) \cap K = \emptyset \quad \text{for all } j \geq J_t \text{ and all } s \in (t-\delta, t+\delta) \cap [0,1].
  \end{equation*}
  We can thus cover $[0,1]$ by finitely many such intervals, and the biggest $J_t$ gives us an upper bound for the set $G_{\tilde{K}}$. Similarly, there is a lower bound for $G_{\tilde{K}}$, so it is finite, and $\Phi$ acts properly.

  \textbf{Claim 3.} \textit{The maps $\phi_0$ and $\phi_1$ are conjugate as elements of $\Diff(\hEins^{1,n})$.}

  By Claim 2, we can construct the quotient manifold $M := [0,1] \times \hEins^{1,n} / \Phi$, and there is a projection $\pi \colon M \to [0,1]$ which is submersive. We show that $M$ is compact. For each $t$, $\pi^{-1}(t) = \hEins^{1,n} / \phi_t$ is compact. Therefore, every open cover $U_\alpha$ of $M$ contains a finite cover $U_{\alpha_i}$ of $\pi^{-1}(t)$. But then, there exsits $\delta > 0$ such that $U_{\alpha_i}$ is also a finite cover of $\pi^{-1}(s)$ for all $s \in (t-\delta, t+\delta)$. Thus, by taking a finite covers of $\pi^{-1}(t)$ for finitely many values of $t$, we can construct a finite cover of $M$.

  Since $M$ is compact, by Ehresmann's theorem, $\pi \colon M \to [0,1]$ is a fiber bundle, and the fibers over different points are diffeomorphic. In particular, there is a diffeomorphism $\bar\chi \colon M_0 \to M_1$, where
  \begin{equation*}
    M_0 := \pi^{-1}(0) = \hEins^{1,n} / \phi_0, \quad M_1:= \pi^{-1}(1) = \hEins^{1,n} / \phi_1.
   \end{equation*}
   By uniqueness of the universal cover, $\bar\chi$ lifts to a diffeomorphism $\chi \colon \Eins^{1,n} \to \Eins^{1,n}$. Conversely, because $\chi$ descends to the quotients, it holds that $\phi_1 = \chi \circ \phi_0 \circ \chi^{-1}$, proving the claim.

   \textbf{End of proof.} We have shown that every escaping element $\phi \in \hOG^0(2,n+1)$ is conjugate to its elliptic part $\phi_E$. Since $\phi_E$ is conjugate to an element of $\R \times \SO(n+1)$ with positive (because it is future-escaping) projection on $\R$, we can join it via a path to $(1,\Id) \in \R \times \SO(n+1)$, that is, to the map
   \begin{equation*}
    T_1 \colon \hEins^{1,n} \longrightarrow \hEins^{1,n}, \quad (t,x) \longmapsto (t+1,x).
   \end{equation*}
   Hence, repeating the argument in Claims 1 to 3, we conclude that $\phi_E$, and hence $\phi$, is conjugate to $T_1$ in $\Diff(\hEins^{1,n})$. By transitivity, we conclude that any two future-escaping elements $\phi, \psi \in \hOG^0(2,n+1)$ are conjugate to each other in $\Diff(\hEins^{1,n})$.
  \end{proof}

\subsection{Conformal transformations of $\hEins^{1,n}$ with a fixed point}

We end this section by discussing how conformal transformations of $\hEins^{1,n}$ with a fixed point correspond 1-to-1 to conformal transformations of Minkoswki spacetime $\R^{1,n}$. This is classical, see e.g.\ \cite[Chap.~9]{PeRi}, but for completeness, we give a proof.

\begin{figure} \label{fig:diam}
 \includegraphics[scale=0.15]{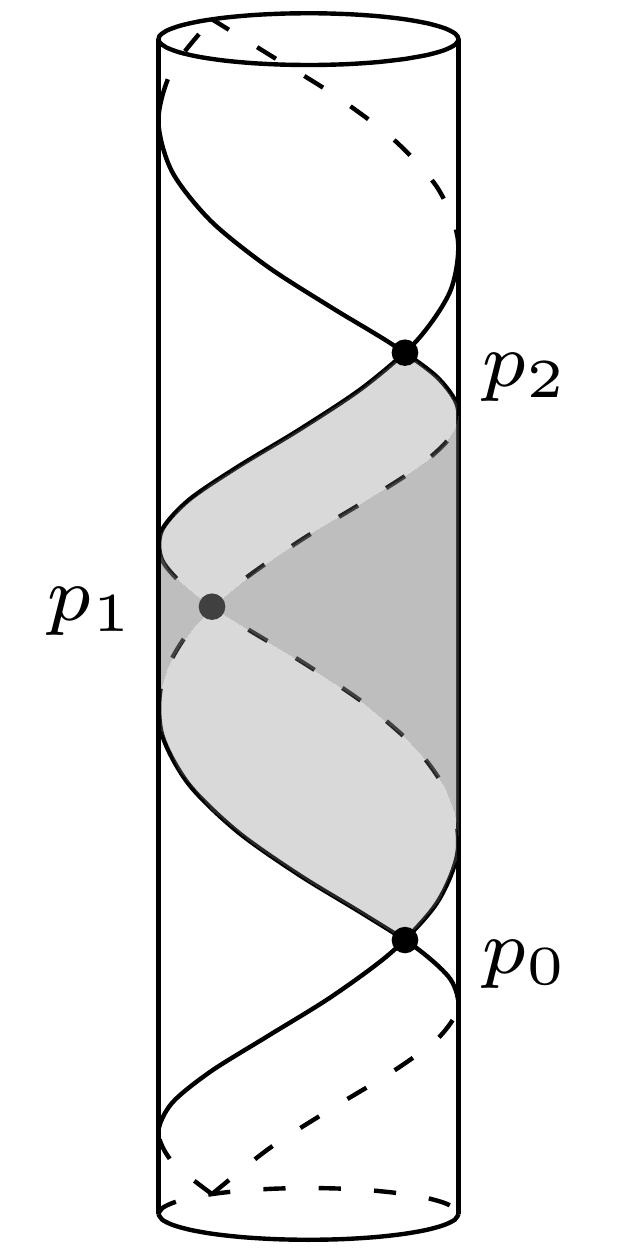}
 \caption{The points $p_0,p_1,p_2 \in \hEins^{1,1}$ as described in Proposition~\ref{prop:fixed}. The spiral-shaped curves are null geodesics going through these points, and the shaded region is the diamond $I(p_0,p_2)$.}
\end{figure}

  \begin{proposition} \label{prop:fixed}
   Let $\phi \in \hOG^\uparrow(2,n+1)$ be a conformal transformation with a fixed point $p = (t,x) \in \hEins^{1,n}$. Then, for all $k \in \Z$, the points
   \begin{equation*}
    p_k := (t + \pi k, \mathrm{sgn}(k) x)
   \end{equation*}
   are fixed, where $\mathrm{sgn}(k)$ is the sign of $k$, and $-x \in \bS^n$ is the antipodal point of $x$. For all $k \in \Z$, $\phi$ preserves the timelike diamond
   \begin{equation*}
    I(p_k,p_{k+2}) := I^+(p_k) \cap I^-(p_{k+2}).
   \end{equation*}
   Moreover, $\phi$ is uniquely determined by the action on one such diamond, acting in the same way on all others.
  \end{proposition}

  Note that a diamond $I(p_k,p_{k+2})$ is conformally diffeomorphic to Minkowski spacetime $\R^{1,n}$. This is the reason for introducing the (compact) Einstein universe as the conformal compactification of Minkowski spacetime. The proposition thus states that elements of $\hOG(1,n+1)$ with a fixed point are (up to specifying the fixed point) the same as conformal transformations of Minkowski. See Figure~\ref{fig:Mink} for an example.

  \begin{proof}
   Consider $\phi$ and $p=(t,x)$ as in the statement. Because geodesics in the product $\R \times \bS^n$ are precisely products of geodesics in the factors, all null geodesics emanating from $p$ towards the future intersect again at $p_1 = (t+\pi,-x)$ for the first time, and then periodically as described by the sequence $(p_k)_k$ (also towards the past). Because null geodesics are preserved by conformal transformations, $\phi$ preserving $p$ implies that the entire sequence $(p_k)_k$ is preserved by $\phi$. Moreover, preservation of the causal relation implies that the timelike diamonds $I(p_k, p_{k+2})$ are preserved. We call this a chain of diamonds. Note that the closures of all diamonds in the chain cover all of $\hEins^{1,n}$.

   Let $F \colon I(p,p_2) \to \R^{1,n}$ be a conformal diffeomorphism. Then $\tilde{\phi} := F \circ \phi \circ F^{-1}$ is an element in $\Conf(\R^{1,n})$. It remains to show that $\tilde{\phi}$ uniquely determines $\phi$. By continuity, knowledge of $\tilde{\phi}$ determines $\phi$ on the boundary $\partial I(p,p_2)$. We show the converse, namely that the behaviour of $\phi$ on $\partial I(p,p_2)$ uniquely determines the behaviour of $\phi$ on the next diamond in the chain, $I(p_1,p_3)$. Then, it follows by induction that $\phi$ is uniquely determined on all of $\hEins^{1,n}$.

   To see this, we need to understand the structure of $\partial I(p_k,p_{k+2})$. Following \cite[Chap.~9]{PeRi}, we have
   \begin{equation*}
    \partial I(p_k,p_{k+2}) = \mathcal{I}_k^- \cup \mathcal{I}_k^+ \cup \{p_k,p_{k+1},p_{k+2}\},
   \end{equation*}
   where $\mathcal{I}_k^+$ are the future endpoints of null geodesics inextendible in $I(p_k,p_{k+2})$, and $\mathcal{I}_k^-$ the past endpoints. Because the diamonds are stacked on top of each other, it holds that $\mathcal{I}_k^+ = \mathcal{I}_{k+1}^-$. Moreover, because $I(p_k,p_{k+2})$ is conformal to Minkowski spacetime, there is a natural identification of $\mathcal{I}_k^+$ and $\mathcal{I}_k^-$, as follows. Every point $q \in \mathcal{I}_k^+$ is uniquely determined by its past $I^-(q)$, which equals $I^-(\gamma)$ for any null geodesic $\gamma$ with endpoint $q$. In Minkowski spacetime, the past of a null geodesic is the region that lies below a lightlike plane. The region above the same lightlike plane corresponds to the future of the same null geodesic, and hence to a unique point in $\mathcal{I}_k^-$.

   Now, a point $x \in I(p,p_2)$ is uniquely determined by either of $A(x) := \partial I^\pm(x) \cap \mathcal{I}^\pm_0$. Therefore, there is a unique point $y \in I(p_1,p_3)$ such that $\partial I^-(y) \cap \mathcal{I}^-_1 = A(x)$. By the natural identification between $\mathcal{I}^+_0$ and $\mathcal{I}^-_0$, $x$ and $y$ represent the same point in Minkowski spacetime. When applying $\phi$, we see that the same is true for $\phi(x)$ and $\phi(y)$, since both are determined uniquely by $\phi(A(x))$. We conclude that the action of $\phi$ on each diamond is the same.
  \end{proof}

 A natural question here is how the Jordan decomposition of elements in $\hOG^\uparrow(2,n+1)$ (or $\OG(2,n+1)$) with a fixed point is related to the different types of conformal transformations of Minkowski spacetime $\R^{1,n}$. The following can be found in \cite[p.~56]{FraPhD}:
 \begin{itemize}
  \item Homotheties $\R^{1,n} \to \R^{1,n},\ x \mapsto c x$ for $c >0$ correspond to hyperbolic elements.
  \item Linear isometries of $\R^{1,n}$ (i.e.\ elements of $\OG(1,n)$) correspond to the stabilizer of a Lorentzian $2$-plane in $\R^{2,n+1}$.
  \item Translations in $\R^{1,n}$ correspond to parabolic elements.
 \end{itemize}

  Finally, we use the chain of diamonds to give an example of non-escaping conformal transformations of $\hEins^{1,n}$ whose product is escaping.

  \begin{example} \label{ex:notsubgroup}
   Let $p_0 = (0,x_0)$ and $q_0 = (\pi,x_0)$ be points in $\hEins^{1,n}$, and let $\phi$ and $\psi$ be the elements in $\hOG^\uparrow(2,n+1)$ that fix $p_0$ and $q_0$ respectively, and act by a time-translation of Minkowski on the corresponding diamonds (see Figure~\ref{fig:Mink}). Then
   \begin{equation*}
    \phi(t,x_0) = (t+a(t),x_0), \quad \psi(t,x_0) = (t+b(t),x_0)
   \end{equation*}
  where $a(t),b(t) \geq 0$ and $a(t) = 0$ only if $t = 2 \pi k$, while $b(t) = 0$ only if $t =\pi + 2 \pi k$ (in both cases, $k \in \Z$). It follows that $\phi \circ \psi$ does not have any fixed points. But then, the sequence $(\phi \circ \psi)^k(p_0)$ cannot converge (because the limit would be a fixed point), so the $t$-coordinate must be unbounded along the sequence, implying that $\phi \circ \psi$ is future-escaping.
  \end{example}

  \begin{figure}
  \centering
   \includegraphics[scale=1.25]{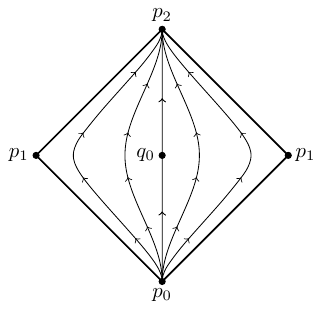}
   \caption{The flowlines of a time-translation of Minkowski $\R^{1,1}$ mapped onto a diamond $I(p_0,p_2)$ in $\hEins^{1,1}$.} \label{fig:Mink}
  \end{figure}

\section{Essentiality} \label{sec:essentiality}

A conformal transformation $\phi \in \Conf^\uparrow(M,g)$ is called \emph{inessential} if there exists a metric $h$ conformal to $g$ (i.e.\ $h = \Omega g$ for a smooth positive function $\Omega$) such that $\phi$ is an isometry for $h$. Otherwise, $\phi$ is called \emph{essential}. Similarly, a subgroup $G \subset \Conf^\uparrow(M,g)$ is called inessential if there exists a metric $h$ conformal to $g$ such that $G \subset \Isom^\uparrow(M,h)$, and essential otherwise.

\subsection{The Lichnerowicz conjecture}

The central question is: Which pseudo-Riemannian manifolds have essential conformal group? In Riemannian signature, Lichnerowicz conjectured that this only happens in highly symmetrical cases. Indeed, Ferrand later proved that if $\Conf^\uparrow(M,g)$ is essential, then $(M,g)$ is conformally diffeomorphic to either the round sphere or Euclidean space. (See also previous work by Ferrand \cite{Fer69} and Obata \cite{Oba} assumping compactness, and later proofs by Schoen \cite{Sch} and Frances \cite{FraENS}.)

In pseudo-Riemannian signatures, the situation is more complicated, and has been investigated in the context of Gromov's theory of rigid geometrical structures \cite{Gro}. In particular, the following conjecture stands out, which has been dubbed the pseudo-Riemannian, or generalized, Lichnerowicz conjecture.

\begin{conjecture}[{\cite[p.~96]{DAGr}}] \label{conj:genLich}
 If $(M,g)$ has essential conformal group, then $(M,g)$ is conformally flat.
\end{conjecture}

The current status of the conjecture is the following:
\begin{itemize}
 \item It is true in Riemannian signature, as discussed above.
 \item There are compact counterexamples due to Frances \cite{FraTG} in all signatures $(p,q)$ with $p,q \geq 2$. There are also locally conformally pseudo-K\"ahler examples by Cort\'es and Leistner \cite{CorLei}.
 \item There are non-compact Lorentzian counterexamples (see \cite[p.~5]{FraESI} and references therein).
 \item The conjecture remains open for compact Lorentzian manifolds, having been proven under additional assumptions by Melnick and Pecastaing \cite{MelPec} and Frances and Melnick \cite{FraMel}.
\end{itemize}

We propose that, in the Lorentzian setting, an alternative to the assumption of compactness of $M$, is to assume global hyperbolicity with compact Cauchy surfaces. A stronger assumption (by Lemma~\ref{lem:NOHCauchy}) is the No Observer Horizons condition (NOH). In this context, Conjecture~\ref{conj:genLich} thus becomes:

\begin{conjecture} \label{conj:Lich1}
 Let $(M,g)$ be a causal spacetime satisfying the NOH. Then, if $\Conf^\uparrow(M,g)$ is essential, $(M,g)$ is conformally flat.
\end{conjecture}

Recall also the stronger conjecture made in the introduction:

\begin{Lconjecture}
 Let $(M,g)$ be a causal spacetime satisfying the NOH. Then $\Conf^\uparrow(M,g)$ is essential if and only if $(M,g)$ is conformally diffeomorphic to the Einstein static universe $\hEins^{1,n}$.
\end{Lconjecture}

While the equivalent to Conjecture~\ref{conj:Lich1} for compact Lorentzian manifolds remains open, the equivalent to Conjecture~\ref{conj:Lich2} was shown to be false by Frances \cite{FraMA,FraESI}, who found an infinite family of non-conformally related compact Lorentzian manifolds with essential conformal group.

\subsection{Essentiality and properness}

The following proposition has a Riemannian analogue, which plays a role in Ferrand's proof of the Lichnerowicz conjecture (see \cite[Thm.~A1]{Fer}).

\begin{proposition} \label{prop:inessential}
 Let $(M,g)$ be a causal spacetime satisfying the NOH, and $G \subset \Conf^\uparrow(M,g)$ a closed subgroup. Then, $G$ is inessential if and only if it acts properly on $M$.
\end{proposition}

\begin{proof}
 Suppose that $G$ is inessential. Then $G \subset \Isom^\uparrow(M,h)$, where $(M,h)$ is a causal spacetime satisfying the NOH (because it is conformal to $(M,g)$). By \cite[Thm.~1]{GaZe}, $\Isom^\uparrow(M,h)$, and hence $G$, act properly on $M$.

 Conversely, suppose that $G$ acts properly. Then, by \cite[Thm.~3.2]{GaZe}, $G$ preserves the gradient $d\tau$ of a Cauchy temporal function $\tau \colon M \to \R$. Then, $h := \Vert d\tau \Vert g$ is a Lorentzian metric that is preserved by $G$, where $\Vert d\tau \Vert$ denotes the pseudo-norm of $d\tau$ with respect to the inverse metric of $g$. It follows that $G$ is inessential.
\end{proof}

Based on this result, Conjecture \ref{conj:Lich1} above is equivalent to saying that if $(M,g)$ is not conformally flat, then $\Conf(M,g)$ acts properly. This is precisely the argument in Ferrand's proof of the Riemannian Lichnerowicz conjecture \cite{Fer}. Moreover, we extract the following corollaries.

\begin{corollary} \label{cor:escap}
 Let $(M,g)$ be a causal spacetime satisfying the NOH. Then all escaping elements of $\Conf^\uparrow(M,g)$ are inessential.
\end{corollary}

\begin{proof}
 Let $\phi$ be escaping. Then $\Gamma := \{ \phi^n \mid n \in \Z \}$ is a discrete (hence closed) subgroup: Otherwise, there would be a converging sequence in $\Gamma$, which would contradict the fact that $\phi$ is escaping. Moreover, $\Gamma$ acts properly, because $\phi$ being escaping implies that for any compact $K \in M$, the set $\Gamma_K := \{\psi \in \Gamma \mid K \cap \psi(K) \neq \emptyset \}$ is finite (hence compact) in $\Gamma$. By Proposition~\ref{prop:inessential}, it now follows that $\Gamma$, and hence $\phi$, are inessential.
\end{proof}

\begin{corollary}
 An element of $\Conf^\uparrow(\hEins^{1,n})$ is inessential if and only if it is escaping or elliptic. Moreover, essential elements always have fixed points.
\end{corollary}

\begin{proof}
 Let $\phi \in \Conf^\uparrow(\hEins^{1,n})$. By Corollary~\ref{cor:escap}, if $\phi$ is escaping, it is inessential. By Lemma~\ref{lem:ellincover}, if $\phi$ is non-escaping elliptic, then $\phi$ is an element in a subgroup $H$ isomorphic to $\OG(n+1)$. Hence $H$ is compact and acts properly, so it is inessential.

 Conversely, suppose that $\phi$ is inessential, so that $\phi \in \Isom^\uparrow(M,g)$ for $(M,g)$ conformally equivalent to $\Eins^{1,n}$. Thus $(M,g)$ is causal and satisfies the NOH, and by \cite[Thm.~1]{GaZe}, $\Isom^\uparrow(M,g) = L \ltimes N$, where $N$ preserves a time function, and elements with non-trivial projection on $L$ are escaping. Hence $\phi$ is either escaping, or it preserves a time function. In the latter case, by Lemma~\ref{lem:ellincover}, $\phi$ must be elliptic.

 Finally, by Theorem~\ref{thm:classification}, if $\phi$ is non-elliptic and non-escaping (hence essential), then $\phi$ must have a fixed point.
\end{proof}

\begin{remark}
 Essentiality in $\hEins^{1,n}$ is different from that in ${\Eins^{1,n}}$. There are inessential elements $\phi \in \hOG(2,n+1)$ such that the projection $\pr \phi \in \OG(2,n+1)$ is essential. For instance, certain escaping elements.
\end{remark}

We have proven that essential conformal transformations of $\hEins^{1,n}$ have fixed points, and that conformal transformations of NOH spacetimes are non-escaping. Given this, the following result provides further evidence for Conjectures~\ref{conj:Lich1} and \ref{conj:Lich2}.

\begin{proposition} \label{prop:achronalset}
 Let $(M,g)$ be a causal spacetime satisfying the NOH, and let $\phi \in \Conf^\uparrow(M,g)$ be non-escaping. Then $\phi$ preserves a non-empty achronal set $S \in M$.
\end{proposition}

\begin{proof}
 Let $p \in M$ be any point, and let $p_k := \phi^k(p)$. Consider the set $V = \bigcup_{k \in \Z} I^-(p_k)$. Being a union of past-sets, also $V$ is a past-set, meaning $V = I^-(V)$. It follows that $\partial V$ is achronal. By construction, $V$, and hence $\partial V$, are invariant under $\phi$. It remains to show that $\partial V \neq \emptyset$; equivalently, that $V \neq M$. Let $\Sigma \in M$ be a (compact) Cauchy surface, and consider the associated orthogonal splitting $M = \R \times \Sigma$. Because $\phi$ is non-escaping, there is a compact set $K = [a,b] \times \Sigma$ such that $p_k \in K$ for all $k \in \Z$. Therefore, $(b,\infty) \times \Sigma \cap V = \emptyset$, and we are done.
\end{proof}

Note that, by the NOH, $S$ is intersected at least once by every inextendible causal curve (exactly once if it is timelike). In particular,
\begin{equation*}
 M = I^-(S) \cup S \cup I^+(S).
\end{equation*}

\subsection{Classification of conformally flat spacetimes}

If Conjecture \ref{conj:Lich1} is true, then Conjecture \ref{conj:Lich2} reduces to classifying certain conformally flat spacetimes. We discuss here some partial answers based on existing literature, but adapted to our context. This requires passing to the universal cover, so we first need to understand what the NOH on a spacetime implies about its universal cover. Recall the following definition from \cite{Ros}.

\begin{definition}
 A \emph{conformal Cauchy embedding} of $M$ into $N$ (both $M$ and $N$ being globally hyperbolic spacetimes) is a conformal embedding $\iota \colon M \to N$ sending every Cauchy surface of $M$ to a Cauchy surface of $N$. A spacetime $M$ is called \emph{C-maximal} if every Cauchy extension is surjective.
\end{definition}

\begin{lemma} \label{lem:NOHcover}
 Let $M$ be a conformally flat causal spacetime satisfying the NOH, and $\tilde{M}$ its universal cover. Then both $M$ and $\tilde{M}$ are C-maximal.
\end{lemma}

\begin{proof}
 By Lemma \ref{lem:NOHCauchy}, $M$ is globally hyperbolic, hence diffeomorphic to $\R \times \Sigma$ for $\Sigma$ a compact Cauchy surface. Let $\iota \colon M \to N$ be a conformal Cauchy embedding, and $V := \iota(M)$. Necessarily, $V$ is causally convex in $N$ and contains a Cauchy surface of $N$. Also $N$ is diffeomorphic to $\R \times \Sigma$, and there are continuous functions $f^\pm \colon \Sigma \to \R \cup \{ \pm \infty \}$ such that
 \begin{equation*}
  V = \{(t,x) \mid f^-(x) < t < f^+(x) \}.
 \end{equation*}
  If $\iota$ is not surjective, then the boundary $\partial V \subset N$ is non-empty, so without loss of generality, there is $x_0 \in \Sigma$ such that $f^+(x_0) < +\infty$. Hence there is a neighborhood $U \subset \partial V$ of $x_0$ such that $f^+ \vert_U < +\infty$. We use this to construct a map
 \begin{equation*}
  b \colon U \longrightarrow \CB^+(M),\quad p \longmapsto I^-(p).
 \end{equation*}
 Note that $I^-(p)$ is a TIP in $V$, because a causal curve with future endpoint $p \in \partial V$ is inextendible in $V$. Because $N$ is globally hyperbolic, it is distinguishing, so $I^-(p) \neq I^-(q)$ for $p \neq q$, implying that $b$ is injective. It follows that $\CB^+(M)$ has infinitely many points, in contradiction to the NOH. Hence $V = N$, so $\iota$ is surjective and $M$ is C-maximal. Finally, by \cite[Cor.~6]{SmaiAIP}, also $\tilde{M}$ is C-maximal (note that C-maximality is equivalent to the property called there $\mathcal{C}_0$-maximality \cite[Thm.~1.2]{SmaiJMP}).
\end{proof}

We now apply a result of Rossi \cite{Ros} to obtain a classification in the case that $\tilde{M}$ has compact Cauchy surfaces, which occurs when the fundamental group of $M$ is finite. The case of infinite fundamental group (implying that $\tilde{M}$ has non-compact Cauchy surfaces) is more complicated, and has been investigated by Sma\"i \cite{SmaiJMP}, but a comparable classification is not yet available.

\begin{theorem} \label{thm:finitepi1}
 Let $(M,g)$ be a causal, conformally flat spacetime satisfying the NOH. Suppose that $\mathrm{dim}(M) = n+1 \geq 3$ and that $M$ has finite fundamental group. Then, up to conformal transformation, $(M,g)$ is a finite quotient of $\hEins^{1,n}$. If $\Conf^\uparrow(M,g)$ contains an essential conformal transformation, then $(M,g)$ is conformally diffeomorphic to $\hEins^{1,n}$.
\end{theorem}

\begin{proof}
 The universal cover $\tilde{M}$ is globally hyperbolic, and by Lemma \ref{lem:NOHcover}, it is C-maximal. We show that $\tilde{M}$ has compact Cauchy surfaces. Then, by \cite[Thm.~9]{Ros}, $\tilde{M}$ is conformally diffeomorphic to $\hEins^{1,n}$, and $M$ is a finite quotient thereof. To see that $\tilde{M}$ has compact Cauchy surfaces, note that $\tilde{M}$ is diffeomorphic to $\R \times \tilde\Sigma$, where $\tilde\Sigma$ is the universal cover of a Cauchy surface $\Sigma \subset M$. Moreover, $\tilde\Sigma$ must have finite fundamental group, so the cover $\tilde\Sigma \to \Sigma$ is finite, and hence $\tilde\Sigma$ is compact.

 It remains to show that if $\Conf^\uparrow(M,g)$ contains an essential element $\phi$, then the cover $\tilde{M} \to M$ is trivial. Let $G := \overline{\langle \phi \rangle}$ denote the closure of the subgroup generated by $\phi$. By uniqueness of the universal cover, $\phi$ lifts to a conformal diffeomorphism $\tilde{\phi} \colon \tilde{M} \to \tilde{M}$, so there is also a $G$-action on $\tilde M$. By Proposition~\ref{prop:inessential}, if $\phi$ acts essentially on $(M,g)$, then $G$ acts non-properly. Hence there is a compact $K \subset M$ such that the set
 \begin{equation*}
  G_{K} := \{ \psi \in G \mid {K} \cap \psi({K}) \neq \emptyset \}
 \end{equation*}
 is non-compact. Because the covering $\pi \colon \tilde{M} \to M$ is finite, $\tilde{K}  := \pi^{-1}(K)$ is compact in $\tilde{M}$. It is then easy to see that $G_{\tilde{K}} = G_K$, and thus $G$ acts non-properly on $\tilde{M}$.

 The rest of the proof consists in arguing that, if the action of $\tilde{\phi}$ is to descend to a finite (non-trivial) quotient of $\tilde{M}$, then $G$ must act properly, in contradiction to the previous paragraph. Suppose thus that $M = \tilde{M} / \Gamma$ for $\Gamma \subset \Conf^\uparrow(\tilde{M},\tilde{g})$ a non-trivial finite subgroup. The map
 \begin{align*}
  c \colon \Gamma &\longrightarrow \Conf^\uparrow(\tilde{M},\tilde{g}) \\ \gamma &\longmapsto \tilde\phi \circ \gamma \circ \tilde\phi^{-1}
 \end{align*}
 defines an automorphism of $\Gamma$. Because $\Gamma$ is finite, so is $\Aut(\Gamma)$, and hence there is a $m \in \N$ such that $c^m = 1$ is the identity. This means that $\tilde\phi^m$ centralizes $\Gamma$ (i.e.\ commutes with all $\gamma \in \Gamma$).

 Since $\tilde{M}$ is conformal to $\hEins^{1,n}$, we have $\Conf^\uparrow(\tilde{M},\tilde{g}) = \hOG(2,n+1)$. Let $\gamma \in \Gamma$ be a non-trivial element. Because $\Gamma$ is a finite (hence compact) subgroup, $\gamma$ is elliptic. Moreover, since $M = \tilde{M} / \Gamma$ is a non-compact manifold, $\gamma$ must be non-escaping and without fixed points. Hence, the projection of $\gamma$ to $\OG(2,n+1)$ must be conjugate to an element of $\OG(n+1)$. Then, in an appropriate basis, $\gamma$ preserves the splitting $\R^{2,n+1} = \R^{2,0} \oplus \R^{0,n+1}$, acting trivially on $\R^{2,0}$ and without fixed points on $\R^{0,n+1}$. Since, by the previous paragraph, $\tilde\phi^m$ commutes with $\gamma$, also $\tilde\phi^m$ must preserve the splitting. Hence the projection of $\tilde\phi^m$ to $\OG(2,n+1)$ is conjugate to an element of $\OG(2) \times \OG(n+1)$. It follows that $\tilde\phi^m$ is elliptic, implying that $\overline{\langle \tilde{\phi}^m \rangle}$ acts properly (see Remark~\ref{rem:ellproper}). Then, by Lemma~\ref{lem:powers}, also $G$ acts properly. This contradicts the previous discussion, so we conclude that $\Gamma = \{e\}$, and $M$ is conformal to $\hEins^{1,n}$.
 \end{proof}

\begin{remark} \label{rem:orbi}
 In the last part of the proof of Theorem~\ref{thm:finitepi1}, it is crucial that the finite group $\Gamma$ acts freely (without fixed points) on $\tilde{M}$. If we allow non-free actions, then the quotient is no longer a manifold, but an orbifold. In that case, the conclusion really fails. For example, in Riemannian signature, the quotient of a sphere by a periodic rotation is such an orbifold. There, the flow from north to south pole along the meridians is an essential transformation, which descends to the quotient by a periodic rotation, and acts essentially also on this quotient orbifold.
\end{remark}

\bibliographystyle{abbrv}
\bibliography{refs.bib}

\begin{thebibliography}{10}

\bibitem{AGZ}
S.~Allout, L.~Garc\'ia-Heveling, and A.~Zeghib.
\newblock Smooth vs algebraic conjugacy for group actions.
\newblock In preparation.

\bibitem{BeSa1}
A.~N. Bernal and M.~S\'anchez.
\newblock Smoothness of time functions and the metric splitting of globally
  hyperbolic spacetimes.
\newblock {\em Comm. Math. Phys.}, 257(1):43--50, 2005.

\bibitem{CorLei}
V.~Cort\'es and T.~Leistner.
\newblock Compact locally conformally pseudo-{K}\"ahler manifolds with
  essential conformal transformations.
\newblock {\em SIGMA Symmetry Integrability Geom. Methods Appl.}, 20:Paper No.
  084, 12, 2024.

\bibitem{DAGr}
G.~D'Ambra and M.~Gromov.
\newblock Lectures on transformation groups: geometry and dynamics.
\newblock In {\em Surveys in differential geometry ({C}ambridge, {MA}, 1990)},
  pages 19--111. Lehigh Univ., Bethlehem, PA, 1991.

\bibitem{EMN}
O.~Eshkobilov, E.~Musso, and L.~Nicolodi.
\newblock On the restricted conformal group of the {$(1+n)$}-{E}instein static
  universe.
\newblock {\em J. Geom. Phys.}, 146:103517, 18, 2019.

\bibitem{Fer}
J.~Ferrand.
\newblock The action of conformal transformations on a {R}iemannian manifold.
\newblock {\em Math. Ann.}, 304(2):277--291, 1996.

\bibitem{FraPhD}
C.~Frances.
\newblock {\em G\'eom\'etrie et dynamique lorentziennes conformes}.
\newblock PhD thesis, ENS Lyon, 2002.
\newblock Available at
  \url{https://irma.math.unistra.fr/~frances/these2-frances.pdf}.

\bibitem{FraMA}
C.~Frances.
\newblock Sur les vari\'et\'es lorentziennes dont le groupe conforme est
  essentiel.
\newblock {\em Math. Ann.}, 332(1):103--119, 2005.

\bibitem{FraENS}
C.~Frances.
\newblock Sur le groupe d'automorphismes des g\'eom\'etries paraboliques de
  rang 1.
\newblock {\em Ann. Sci. \'Ecole Norm. Sup. (4)}, 40(5):741--764, 2007.

\bibitem{FraESI}
C.~Frances.
\newblock Essential conformal structures in {R}iemannian and {L}orentzian
  geometry.
\newblock In {\em Recent developments in pseudo-{R}iemannian geometry}, ESI
  Lect. Math. Phys., pages 231--260. Eur. Math. Soc., Z\"urich, 2008.

\bibitem{FraTG}
C.~Frances.
\newblock About pseudo-{R}iemannian {L}ichnerowicz conjecture.
\newblock {\em Transform. Groups}, 20(4):1015--1022, 2015.

\bibitem{FraMel}
C.~Frances and K.~Melnick.
\newblock The {L}orentzian {L}ichnerowicz conjecture for real-analytic,
  three-dimensional manifolds.
\newblock {\em J. Reine Angew. Math.}, 803:183--218, 2023.

\bibitem{GaZe}
L.~Garc\'{i}a-Heveling and A.~Zeghib.
\newblock Isometries of spacetimes without observer horizons, 2025.
\newblock Preprint arXiv:2502.13904.

\bibitem{Ger}
R.~Geroch.
\newblock Domain of dependence.
\newblock {\em J. Mathematical Phys.}, 11:437--449, 1970.

\bibitem{Gro}
M.~Gromov.
\newblock Rigid transformations groups.
\newblock In {\em G\'eom\'etrie diff\'erentielle ({P}aris, 1986)}, volume~33 of
  {\em Travaux en Cours}, pages 65--139. Hermann, Paris, 1988.

\bibitem{HawAdams}
S.~Hawking.
\newblock {Singularities and the geometry of space-time}.
\newblock {\em EPJ H}, 39:413--–503, 2014.

\bibitem{Hel}
S.~Helgason.
\newblock {\em Differential geometry, {L}ie groups, and symmetric spaces},
  volume~80 of {\em Pure and Applied Mathematics}.
\newblock Academic Press, Inc. [Harcourt Brace Jovanovich, Publishers], New
  York-London, 1978.

\bibitem{HiOl}
J.~Hilgert and G.~\'Olafsson.
\newblock {\em Causal symmetric spaces}, volume~18 of {\em Perspectives in
  Mathematics}.
\newblock Academic Press, Inc., San Diego, CA, 1997.

\bibitem{Hum}
J.~E. Humphreys.
\newblock {\em Linear algebraic groups}, volume No. 21 of {\em Graduate Texts
  in Mathematics}.
\newblock Springer-Verlag, New York-Heidelberg, 1975.

\bibitem{Knapp}
A.~W. Knapp.
\newblock {\em Lie groups beyond an introduction}, volume 140 of {\em Progress
  in Mathematics}.
\newblock Birkh\"auser Boston, Inc., Boston, MA, second edition, 2002.

\bibitem{Lee}
J.~M. Lee.
\newblock {\em Introduction to smooth manifolds}, volume 218 of {\em Graduate
  Texts in Mathematics}.
\newblock Springer, New York, second edition, 2013.

\bibitem{Fer69}
J.~Lelong-Ferrand.
\newblock Transformations conformes et quasiconformes des vari\'et\'es
  riemanniennes; application \`a{} la d\'emonstration d'une conjecture de {A}.
  {L}ichnerowicz.
\newblock {\em C. R. Acad. Sci. Paris S\'er. A-B}, 269:A583--A586, 1969.

\bibitem{Mal}
D.~B. Malament.
\newblock The class of continuous timelike curves determines the topology of
  spacetime.
\newblock {\em J. Math. Phys.}, 18(7):1399--1404, 07 1977.

\bibitem{MelPec}
K.~Melnick and V.~Pecastaing.
\newblock The conformal group of a compact simply connected {L}orentzian
  manifold.
\newblock {\em J. Amer. Math. Soc.}, 35(1):81--122, 2022.

\bibitem{Oba}
M.~Obata.
\newblock The conjectures on conformal transformations of {R}iemannian
  manifolds.
\newblock {\em J. Differential Geometry}, 6:247--258, 1971/72.

\bibitem{Paeng}
S.-H. Paeng.
\newblock Hawking's singularity theorem under a bounded integral norm of
  {R}icci curvature.
\newblock {\em Journal of Geometry and Physics}, 183:104708, 2023.

\bibitem{PeRi}
R.~Penrose and W.~Rindler.
\newblock {\em Spinors and space-time. {V}ol. 2}.
\newblock Cambridge Monographs on Mathematical Physics. Cambridge University
  Press, Cambridge, 1986.

\bibitem{Ros}
C.~Rossi~Salvemini.
\newblock Maximal extension of conformally flat globally hyperbolic
  space-times.
\newblock {\em Geom. Dedicata}, 174:235--260, 2015.

\bibitem{Sch}
R.~Schoen.
\newblock On the conformal and {CR} automorphism groups.
\newblock {\em Geom. Funct. Anal.}, 5(2):464--481, 1995.

\bibitem{SchCFT}
M.~Schottenloher.
\newblock {\em A mathematical introduction to conformal field theory}, volume
  759 of {\em Lecture Notes in Physics}.
\newblock Springer-Verlag, Berlin, second edition, 2008.

\bibitem{SmaiJMP}
R.~Sma\"i.
\newblock Causal completion of globally hyperbolic conformally flat spacetimes.
\newblock {\em J. Math. Phys.}, 65(10):Paper No. 102501, 19, 2024.

\bibitem{SmaiAIP}
R.~Smaï.
\newblock Enveloping space of a globally hyperbolic conformally flat spacetime,
  2023.
\newblock Preprint arXiv:2311.17802, to appear in Annales de l'Institut
  Fourier.

\end{thebibliography}

\end{document}